\documentclass{amsart}
\usepackage{amsmath}
\usepackage{amsfonts}
\usepackage{amsthm}
\usepackage{amssymb}
\usepackage{enumerate}
\usepackage{cite}
\theoremstyle{plain}
\newtheorem{thm}{Theorem}[section]

\newtheorem{proposition}[thm]{Proposition}
\newtheorem{exam}[thm]{Example}

\newtheorem{obser}[thm]{Observation}

\newtheorem{lemma}[thm]{Lemma}
\theoremstyle{definition}
\newtheorem{defn}[thm]{Definition}
\theoremstyle{remark}
\newtheorem{rmk}[thm]{Remark}

\makeatletter

\newcommand{\Rmnum}[1]{\expandafter\@slowromancap\romannumeral #1@}
\makeatother
\thanks{The first author is supported by NSF of China No.11671057 and No.11688101}

\begin{document}
\title{The CAT(0) geometry of convex domains with the Kobayashi metrics}
\author{Jinsong Liu $\&$ Hongyu Wang}
\address{HLM, Academy of Mathematics and Systems Science,
Chinese Academy of Sciences, Beijing, 100190, China $\&$ School of
Mathematical Sciences, University of Chinese Academy of Sciences,
Beijing, 100049, China} \email{liujsong@math.ac.cn \:wanghongyu16@mails.ucas.ac.cn}

\begin{abstract} {Let $(\Omega,K_{\Omega})$ be a convex domain in $\mathbb C^d$ with the Kobayashi metric $K_{\Omega}$. In this paper we prove that $m$-convexity is a necessary condition for $(\Omega, K_{\Omega})$ to be CAT(0) if $d=2$. Moreover, when $\Omega \subset \mathbb{C}^d, \: d\geq3$, we obtain a similar result with the further smoothness assumption on its boundary. }
\end{abstract}
\maketitle
\section{\noindent{{\bf Introduction}}}
A CAT(0) space is a geodesic metric space whose geodesic triangles are “slimmer”
than the corresponding flat triangles in the Euclidean plane $\mathbb{R}^2$. CAT(0) spaces are natural generalizations of complete simply connected manifolds with nonpositive sectional curvature. Refer to \cite{Bridson} for more details.

\bigskip
Our main result is the following.
\begin{thm}\label{main}
(1) \: If $\Omega\subset \mathbb{C}^2$ is a $\mathbb{C}$-proper convex domain and $(\Omega,K_{\Omega})$ is CAT(0), then $\Omega$ is locally m-convex for some $m\in\mathbb{N}$.

(2) \: Suppose that $\Omega\subset \mathbb{C}^d\: (d\geq2)$ is a bounded convex domain with smooth boundary. If $(\Omega, \:K_{\Omega})$ is CAT(0), then $\partial\Omega$ has finite line type.
\end{thm}

Recall that a convex domain is called $\mathbb{C}$-proper if $\Omega$ does not contain any entire complex affine lines. There is a well-known result on $\mathbb{C}$-proper convex domains.
\begin{proposition}[\cite{Barth1980Convex}]
If $\Omega$ is a $\mathbb{C}$-proper convex domain in $\mathbb{C}^d$, then the Kobayashi metric $K_{\Omega}$ is complete.
\end{proposition}
A $\mathbb{C}$-proper convex domain is called locally m-convex if for any $R>0$ there exists $C>0$ and $m\geq 1$ such that, for all $z\in B(0,R)\cap \Omega$ and non-zero $v\in \mathbb{C}^d$,
$$\delta_{\Omega}(z,v)\leq C\delta_{\Omega}^{\frac{1}{m}}(z).$$

Note that $m$-convexity is related to finite type by the following proposition.
\begin{proposition}[\cite{Zimmer2016Gromov}, \:Proposition 9.1]
Given a bounded convex domain $\Omega$ with smooth boundary, then $\Omega$ is $m$-convex for some $m\in\mathbb{N}$ if and only if $\partial\Omega$ has finite line type in the sense of D'Angelo.
\end{proposition}

\bigskip
\subsection{Motivation from Gromov Hyperbolicity}
\begin{defn}
Let $(X, d)$ be a metric space. Given three points $x, y, o \in$ $X,$ the Gromov product is given by
$$(x | y)_{o}=\frac{1}{2}(d(x, o)+d(o, y)-d(x, y)).$$
A proper geodesic metric space $(X, d)$ is Gromov hyperbolic if and
only if there exists $\delta \geq 0$ such that, for all $o, x, y, z
\in X$,
$$(x | y)_{o} \geq \min \left\{(x | z)_{o},(z | y)_{o}\right\}-\delta.$$
\end{defn}
Z~M. Balogh and M.~Bonk \cite{balogh2000gromov} firstly proved those strongly pseudoconvex domains equipped with the Kobayashi metric are Gromov hyperbolic. Later A~M. Zimmer \cite{Zimmer2016Gromov} proved that smooth convex domains with the Kobayashi metrics are Gromov hyperbolic if and only if they are of finite type.

Recently Zimmer proved that locally $m$-convexity is a necessary condition for those convex domains to be Gromov hyperbolicity.
\begin{thm}[\cite{Zimmer2019subellip}, Corollary 7.2]
Suppose that $\Omega$ is a $\mathbb{C}$-proper convex domain and $(\Omega,K_{\Omega})$ is Gromov hyperbolicity. Then $\Omega$ is locally $m$-convex.
\end{thm}
This paper is motivated by the above Zimmer's work, and Theorem $\ref{main}$ can be seen as an analogue of the above Theorem 1.5.
\bigskip
\section{{\bf Preliminaries}}
\subsection{Notations}\

(1) \:For $z\in\mathbb{C}^d$, let $|z|$ be the standard Euclidean
norm and let $d_{euc}(z_1,z_2)= |z_1-z_2 |$ be the standard Euclidean
distance.

(2) \:Given an open set $\Omega\subset\mathbb{C}^n,\:p\in\Omega$ and
$v\in\mathbb{C}^n\backslash\{0\}$, let
$$\delta_{\Omega}(p)=\inf\{d_{euc}(p,x):x\in\partial \Omega\}$$
as before, and let
$$\delta_{\Omega}(p,v)=\inf\{d_{euc}(p,x):x\in\partial \Omega\cap(p+\mathbb{C}v)\}.$$

(3) \:For any curve $\sigma$, we denote by $L(\sigma)$ the
length of $\sigma$.

(4) \:For any $z_0 \in \mathbb C^n$ and $\delta >0$, let $B_{euc}(z_0, \delta)$ denote the open ball
$B_{euc}(z_0, \delta)=\{z\in \mathbb C^n| \:|z-z_0|<\delta\}$.

(5) \:Write $Aff(\mathbb{C}^d)$ the group of complex affine automorphisms of $\mathbb{C}^d$.

(6) \:Let $\mathbb{X}_d$ denote the set of all $\mathbb{C}$-proper convex domains in $\mathbb{C}^d$ endowed with the local Hausdorff topology.
\subsection{The Kobayashi metric}\
Given a domain $\Omega \subset \mathbb{C}^{d}$, the (infinitesimal)
Kobayashi metric is the pseudo-Finsler metric defined by
$$k_{\Omega}(x ; v)=\inf \left\{|\xi|: f \in \operatorname{Hol}(\mathbb{D}, \Omega), \: f(0)=x, \:
d(f)_{*,0}(\xi)=v\right\}.$$ Define the length of any curve $\sigma$
to be
$$L(\sigma)=\int_{a}^{b} k_{\Omega}\left(\sigma(t) ; \sigma^{\prime}(t)\right) d
t.$$ Then we can define the Kobayashi pseudo-distance to be
\begin{align*}
K_{\Omega}(x, y)&=\inf \left\{L(\sigma)| \:\sigma :[a, b]
\rightarrow \Omega \text { is any absolutely continuous curve }\right.\\
& \text { with } \sigma(a)=x \text { and } \sigma(b)=y \}.
\end{align*}

The following is a well known property on the Kobayashi metric.
\begin{proposition}
If $f: \Omega_{1} \rightarrow \Omega_{2}$ is holomorphic, then, for all $z \in \Omega_{1}$ and $v \in \mathbb{C}^{d}$,
$$
k_{\Omega_{2}}\left(f(z); \: d f_{z}(v)\right) \leq k_{\Omega_{1}}(z ; \: v).
$$
Moreover,
$$
K_{\Omega_{2}}\left(f\left(z_{1}\right), f\left(z_{2}\right)\right) \leq K_{\Omega_{1}}\left(z_{1}, z_{2}\right),
$$
for all $z_{1}, z_{2} \in \Omega_{1}$.
\end{proposition}
For any product domain, the Kobayashi metric has the following product property (cf. \cite{Jarnicki1993Invariant}, p.107),
$$K_{\Omega_1\times\Omega_2}((x_1,y_1),(x_2,y_2))=\max\{K_{\Omega_1}(x_1,x_2),K_{\Omega_2}(y_1,y_2)\},$$
which makes a product domain behave like a positively curved space.

\subsection{CAT(0) space}
\begin{defn}
Let $I\subset\mathbb{R}$ be an
interval. A map $\sigma: I\rightarrow \Omega$ is called a geodesic segment
if, for all $s,t\in I$,
$$K_{\Omega}(\sigma(s),\sigma(t))=|t-s|.$$
And $(X,d)$ is called a geodesic metric space if any two points in $X$ are joined by a geodesic segment.
\end{defn}
\begin{rmk}
Note that the paths which are commonly called 'geodesics' in differential
geometry need not be geodesics in the above sense. In general they will only
be local geodesics.
\end{rmk}

Let $(X,d)$ be a geodesic metric space. For any three points $a,b,c\in X$,
suppose that $[a, b],[b, c],[c, a]$ form a geodesic triangle $\Delta$. Let $\bar{\Delta}(\bar{a}, \bar{b}, \bar{c}) \subset \mathbb{R}^{2}$ be a triangle in the Euclidean plane with the same edge lengths as $\Delta$. Let $p, q$ be any points on $[a, b]$ and
$[a, c],$ and let $\bar{p}, \bar{q}$ be the corresponding points on $[\bar{a}, \bar{b}]$ and $[\bar{a}, \bar{c}],$ respectively, such that
$$
\operatorname{dist}_{X}(a, p)=
\operatorname{dist}_{\mathbb{R}^{2}}(\bar{a}, \bar{p}), \quad \operatorname{dist}_{X}(a, q)=\operatorname{dist}_{\mathbb{R}^{2}}(\bar{a}, \bar{q}) .
$$
\begin{defn}
We call $(X, d)$ a $\mathbf{C A T}(\mathbf{0})$ space, if for any geodesic triangle $\Delta\subset X$ the
inequality $\operatorname{dist}_{X}(p, q) \leq \operatorname{dist}_{\mathbb{R}^{2}}(\bar{p}, \bar{q})$ holds.
\end{defn}
Typical examples are trees and complete
simply connected manifolds with non-positive sectional curvature.
Note that there is an equivalent definition about $CAT(0)$ spaces.
\begin{thm}[\cite{bruhat1972groupes}]\label{mid}
Let $(X,d)$ be a geodesic metric space.
Then $(X,d)$ is $CAT(0)$ if and only if for any three points $x,y,z\in X$,
$$d^2(z,m)\leq\frac{1}{2}(d^2(z,x)+d^2(z,y))-\frac{1}{4}d^2(x,y),$$
where $m$ is the midpoint of the geodesic segment from $x$ to $y$.
\end{thm}

\subsection{Finite type}
For any function $f: \mathbb{C} \rightarrow \mathbb{R}$ with $f(0)=0$, we will denote by $\nu(f)$ the
order of vanishing of $f$ at $0$.

Let $\Omega=\left\{z \in \mathbb{C}^{d}: r(z)<0\right\}$ where $r$ is a
$C^{\infty}$ defining function with $\nabla r \neq 0$ near $\partial \Omega$. A point $x \in \partial \Omega$ is said to have finite line type
$L$ if
$$
\sup \left\{\nu(r \circ \ell) | \ell: \mathbb{C} \rightarrow \mathbb{C}^{d} \text { is a non-trivial affine map and } \ell(0)=x\right\}=L
$$
Note that $\nu(r \circ \ell) \geq 2$ if and only if $\ell(\mathbb{C})$ is tangent to $\Omega .$
\subsection{Local Hausdorff topology}

Given a set $A\subset \mathbb{C}^d$,
let $\mathcal{N}_{\epsilon}$ denote the $\epsilon$-neighborhood of $A$ with respect to
the Euclidean distance. The Hausdorff distance between any two compact sets $A$,$B$ is
given by
$$
d_{H}(A,B)=\inf\{\epsilon>0:A\subset\mathcal{N}_{\epsilon}(B) \quad \text{and} \quad B\subset\mathcal{N}_{\epsilon}(A)\}.
$$
The Hausdorff distance is a complete metric on the space of compact sets in $\mathbb{C}^d$.
The space of all closed convex sets in $\mathbb{C}^d$ can be given a topology from the local
Hausdorff semi-norms.

For $R>0$ and a set $A\subset\mathbb{C}^d$, let
$A^{(R)}:=A\cap B_{R}(0)$. Then we can define the local Hausdorff semi-norms by
$$
d_{H}^{(R)}(A, \:B):=d_{H}(A^{(R)}, \: B^{(R)}).
$$
Since an open convex set is completely determined by its closure, we say a sequence of open convex sets $\{A_n\}$ converges in the local Hausdorff topology to an open convex set $B$ if $d_{H}^{(R)}(\overline{A}, \: \overline{B})\rightarrow 0$ for all $R>0$.

\bigskip
Recently A M. Zimmer proved the following result.
\begin{thm}[\cite{Zimmer2016Gromov},Theorem 4.1]\label{hausdorff}
Suppose that $\{\Omega_n\}$ is a sequence of $\mathbb{C}$-proper convex domains converging to a $\mathbb{C}$-proper convex domain $\Omega$ in the local Hausdorff topology. Then, for all $x, y\in\Omega$,
$$
K_{\Omega}(x,y)=\lim\limits_{n\rightarrow\infty} K_{\Omega_n}(x,y),
$$
uniformly on compact sets of $\Omega\times\Omega$.
\end{thm}

\bigskip
\section{{\bf Proof of Theorem \ref{main} }}
Our proof is based on the following simple observation.
\begin{obser}
If $\Omega=\Omega_1\times\Omega_2$ is a Kobayashi hyperbolic domain, then $(\Omega,K_{\Omega})$ is not a $CAT(0)$ space.
\end{obser}
\begin{proof}
Take $x\neq y\in\Omega_1$ and let $m$ be the midpoint of the geodesic segment from $x$ to $y$ in $(\Omega_1,K_{\Omega_1})$. Then, we can choose $z, \:w\in\Omega_2$ such that
$$
K_{\Omega_2}(z,w)=K_{\Omega_1}(x,m)=K_{\Omega_1}(y,m)=\frac{1}{2}K_{\Omega_1}(x,y),
$$
which implies that
$$
\frac{1}{2}(K_{\Omega}^2((x,w),(m,z))+K_{\Omega}^2((y,w),(m,z)))-\frac{1}{4}K_{\Omega}^2((x,w),(y,w))=0.
$$
Since
$K_{\Omega}((m,w),(m,z))>0$, it follows from Theorem \ref{mid} that $(\Omega,K_{\Omega})$ is not $CAT(0)$.
It completes the proof.
\end{proof}
To prove Theorem $\ref{main}$, we shall need a recent result due to A M. Zimmer.

\begin{thm}[\cite{Zimmer2019subellip}, Theorem 6.1]\label{scaling}
Suppose that $\Omega$ is a $\mathbb{C}$-proper convex domain and every domain in $\overline{Aff(\mathbb{C}^d)\cdot\Omega}\bigcap\mathbb{X}_d$ does not contain any affine disk in the boundary. Then $\Omega$ is locally m-convex for some $m\geq 1$.
\end{thm}

The above theorem shows that: if $\Omega$ is not m-convex, then by scaling we can find an affine disk in the boundary. By using the above theorem, the next Lemma is obvious.

\begin{lemma}\label{scaling2}
Let $\Omega$ be a $\mathbb{C}$-proper convex domain. If $\Omega$ is not locally $m$-convex for any $m\in\mathbb{N}$, then there exists ${A_n}\in Aff(\mathbb{C}^d)$ such that $A_n\Omega\rightarrow \widehat{\Omega}$ and $\widehat{\Omega}\supseteq C(\alpha,\beta)\times\Delta\times\{\vec{0}\}$, where $C(\alpha,\beta)=\{z\in\mathbb{C}:\arg z\in(\alpha,\beta)\}$ is a convex cone.
\end{lemma}
\begin{proof}
In view of Theorem $\ref{scaling}$, we may assume that $\{0\}\times\Delta\times\{\vec{0}\}\subset\partial\Omega$, and
$$
\Omega\subset\{(z_1,...,z_d):Im z_1>0\}.
$$
Writing
$$A_n(z)=\left(
 \begin{array}{cc}
 n & 0 \\
 0& I_{d-1}\\
 \end{array}
 \right),
$$
we obtain
$$
A_n(\Omega\cap\mathbb{C}\times\{\vec{0}\})=C(\alpha,\beta)\times\{\vec{0}\},
$$
where $C(\alpha,\beta)=\bigcup\limits_{t>0}t(\Omega\cap\mathbb{C}\times\{\vec{0}\})$.
\end{proof}
\begin{lemma}\label{isom}
Let $\Omega\subset \mathbb{C}^d$ be a $\mathbb{C}$-proper convex domain. Suppose that $P:\mathbb{C}^d\rightarrow \mathbb{C}$ is the projection map $P(z_1,...z_d)=z_1$. If $\Omega\cap (\mathbb{C}\times\{\vec{0}\})=U\times\{\vec{0}\}$ and $P(\Omega)=U$, then the map $F:U\rightarrow \Omega$ given by $F(z)=(z,\vec{0})$ induces an isometric embeddding $(U,K_{U})\rightarrow (\Omega,K_{\Omega})$.
\end{lemma}
\begin{proof}
Since both $F$ and $P$ are holomorphic maps, from the distance decreasing property of the Kobayashi metrics, it follows that
$$
K_{\Omega}(F(z_1),F(z_2))\leq K_{U}(z_1,z_2).
$$
Noting that $P\circ F=id$, we thus have
$$
K_{U}(z_1,z_2)\leq K_{\Omega}(F(z_1),F(z_2)).
$$
\end{proof}

Now we are ready to prove Theorem $\ref{main}$
\begin{proof}
Part {\bf (1)}. We shall first prove the theorem when $\Omega\subset\mathbb{C}^2$ is a $\mathbb{C}$-proper convex domain.

Assume, by contradiction, that $\Omega$ is not locally $m$-convex. From Lemma $\ref{scaling2}$, it follows that there exists $A_n\in Aff(\mathbb{C}^d)$ such that $\Omega_n:= A_n\Omega\rightarrow \widehat{\Omega}$ and $\widehat{\Omega}\supseteq C(\alpha,\beta)\times\Delta$ and $\widehat{\Omega}\cap\mathbb{C}=C(\alpha,\beta)$.

We claim that $P(\widehat{\Omega})=C(\alpha,\beta)$, where $P(z_1,z_2)=z_1$ is the projection map.
Suppose that it is not the case. Take $p=(z,\omega)\in\widehat{\Omega}$, where $z$ is not contained in $C(\alpha,\beta)$ and $\omega=|\omega|\: e^{i\theta}$. And take $q=(\xi,-e^{i\theta})$ where $Im\:\xi=Im z$ and $\xi$ lies in the boundary of $C(\alpha,\beta)$ such that $Re \xi \cdot Re z>0$.
Since $\widehat{\Omega}$ is also convex, it implies that
$$
\{tp+(1-t)q:t\in(0,1)\}\subset\widehat{\Omega}.
$$
By taking $t=\frac{1}{|\omega|+1}$, we obtain that $tz+(1-t)\xi\in\widehat{\Omega}$, which contradicts with the fact that $\widehat{\Omega}\cap\mathbb{C}=C(\alpha,\beta)$.

Then, by using Lemma $\ref{isom}$, it follows that the map $f:C(\alpha,\beta)\rightarrow\widehat{\Omega}$ given by $f(z)=(z,0)$ induces an isometric embeddding
$$
(C(\alpha,\beta),K_{C(\alpha,\beta)})\rightarrow (\widehat{\Omega},K_{\widehat{\Omega}}).
$$
Now we choose $x,y\in C(\alpha,\beta)$ and let $m$ be the midpoint of the geodesic segment from $x$ to $y$ in the metric space $(C(\alpha,\beta),K_{C(\alpha,\beta)})$. Since $f$ is isometric, $m$ is also the midpoint of the geodesic segment from $x$ to $y$ in metric space $(\widehat{\Omega},K_{\widehat{\Omega}})$.
Therefore, we can take $z\in\Delta$ such that
$$
K_{\Delta}(0,z)=K_{C(\alpha,\beta)}(x,m)=K_{C(\alpha,\beta)}(m,y)=\frac{1}{2}K_{C(\alpha,\beta)}(x,y).
$$
Denote $C=C(\alpha,\beta)\times\Delta$, $\hat{x}=(x,0)$, $\hat{y}=(y,0)$, $\hat{m}=(m,0)$ and $\hat{z}=(0,z)$.

Since $\widehat{\Omega}\supseteq C$, it follows that $$K_{C}(\hat{x},\hat{m})\geq K_{\widehat{\Omega}}(\hat{x},\hat{m}),$$ and $$K_{C}(\hat{y},\hat{m})\geq K_{\widehat{\Omega}}(\hat{y},\hat{m}).$$
Therefore,
\begin{align*}
&\frac{1}{2}(K_{\widehat{\Omega}}^2(\hat{x},\hat{m})+K_{\widehat{\Omega}}^2(\hat{y},\hat{m}))
-\frac{1}{4}K_{\widehat{\Omega}}^2(\hat{x},\hat{y})\\
\leq &\frac{1}{2}(K_{C}^2(\hat{x},\hat{z})+K_{C}^2(\hat{y},\hat{z}))-
\frac{1}{4}K_C^2(\hat{x},\hat{y})\\
=&0.
\end{align*}

Choose $x_n,y_n,z_n\in\Omega$ such that $A_nx_n\rightarrow\hat{x}$, \:$A_ny_n\rightarrow\hat{y}$ and $A_nz_n\rightarrow\hat{z}$. Now Theorem $\ref{hausdorff}$ gives
$$
K_{\widehat{\Omega}}(\hat{x},\hat{y})=\lim\limits_{n\rightarrow\infty}K_{\Omega_n}(A_nx_n,A_ny_n),
$$
and
$$
K_{\widehat{\Omega}}(\hat{x},\hat{z})=\lim\limits_{n\rightarrow\infty}K_{\Omega_n}(A_nx_n,A_nz_n),
$$
and
$$
K_{\widehat{\Omega}}(\hat{y},\hat{z})=\lim\limits_{n\rightarrow\infty}K_{\Omega_n}(A_ny_n,A_nz_n).
$$
Let $m_n$ be the midpoint of the geodesic segment from $A_nx_n$ to $A_ny_n$ in $(\Omega_n,K_{\Omega_n})$.\\
Then, by choosing a subsequence (still denoted by $m_n$), we may suppose that $m_n\rightarrow \hat{m}\in\widehat{\Omega}\cup\{\infty\}$.
Then either $\hat{m}\neq\hat{z}$ or $\hat{m}\neq\check{z}$, where $\check{z}=(m,iz)$.

Since $K_{\Delta}(0,z)=K_{\Delta}(0,iz)$,
the equalities
$K_{C}(\hat{x},\hat{z})=K_{}(\hat{x},\check{z})$ and
$K_{C}(\hat{y},\hat{z})=K_{C}(\hat{y},\check{z})$ follow.
We have thus proved that
\begin{align*}
&\frac{1}{2}(K_{\widehat{\Omega}}^2(\hat{x},\check{z})+K_{\widehat{\Omega}}^2(\hat{y},\check{z}))
-\frac{1}{4}K_{\widehat{\Omega}}^2(\hat{x},\hat{y})\\
&\leq\frac{1}{2}(K_{C}^2(\hat{x},\check{z})+K_{C}^2(\hat{y},\check{z}))-
\frac{1}{4}K_{C}^2(\hat{x},\hat{y})\\
&=0.
\end{align*}
Therefore, we deduce that: $\forall\epsilon>0$, there exists $N\in\mathbb{N}$ such that $\forall n>N$
$$
\frac{1}{2}(K_{\Omega_n}^2(\hat{x},\hat{z})+K_{\Omega_n}^2(\hat{y},\hat{z}))
-\frac{1}{4}K_{\Omega_n}^2(\hat{x},\hat{y})\leq \epsilon,
$$
and
$$
\frac{1}{2}(K_{\Omega_n}^2(\hat{x},\check{z})+K_{\Omega_n}^2(\hat{y},\check{z}))
-\frac{1}{4}K_{\Omega_n}^2(\hat{x},\hat{y})\leq \epsilon.
$$
Combining with the fact that $\hat{m}\neq\hat{z}$ or $\hat{m}\neq\check{z}$, we have thus proved that there exists $\delta>0$ such that one of $K_{\Omega_n}(m_n,\hat{z})$ and $K_{\Omega_n}(m_n,\check{z})$ is strictly bigger than $\delta$.
Therefore, in terms of the definition of $CAT(0)$ spaces, by choosing $\epsilon$ small enough, we complete the proof.

\bigskip
Part {\bf (2)}. \:Next we prove the result for the general case  that $\Omega \subset \mathbb C^d, \: d\geq3$, is a bounded convex domain with smooth boundary. The difference is that when $d\geq 3$, the claim $P(\widehat{\Omega})=C(\alpha,\beta)$ may be not correct without the further smoothness assumption on the boundary.

We will use the proof of Proposition 6.1 in \cite{Zimmer2016Gromov}. For the sake of completeness, we present its proof here.

Suppose $\vec{0}\in\partial\Omega$ and
$$
\Omega \cap \mathcal{O}=\left\{\vec{z} \in \mathcal{O}: \operatorname{Im}\left(z_{1}\right)>f\left(\operatorname{Re}\left(z_{1}\right), z_{2}, \ldots, z_{d}\right)\right\},
$$
where $\mathcal{O}$ is a neighborhood of the origin and $f: \mathbb{R} \times \mathbb{C}^{d-1} \rightarrow \mathbb{R}$ is a smooth convex non-negative function. Assuming that $\vec{0}$ has infinite line type, by changing the coordinates if necessary, we have
$$
\lim _{z \rightarrow 0} \frac{f(0, z, 0, \ldots, 0)}{|z|^{n}}=0.
$$
Then there are two cases (a) (b): \\
(a). If $\partial\Omega$ contains an affine disk at $\vec{0}$, without losing of generality we assume that $\vec{0}\times\Delta\times\{\vec{0}\}\subset\partial\Omega$. By taking
$$A_n(z)=\left(
 \begin{array}{cc}
 n & 0 \\
 0& I_{d-1}\\
 \end{array}
 \right),
$$
we deduce that $A_n(\Omega)\rightarrow\widehat{\Omega}$, and
$$
\mathbb{H}\times\Delta\times\{\vec{0}\}\subset\widehat{\Omega},
$$
where $\mathbb{H}$ is the upper half plane.

Since $\widehat{\Omega}\subset\{z\in\mathbb{C}^d:Im z_1>0\}$, by considering the projection
$P:\mathbb{C}^d\rightarrow\mathbb{C}^1$, $P(z_1,...,z_d)=z_1$, we obtain
$$
P(\widehat{\Omega})=\mathbb{H}.
$$
Therefore, the map $f:\mathbb{H}\rightarrow\widehat{\Omega}$ given by $f(z)=(z,\vec{0})$ induces an isometric embeddding $(\mathbb{H},K_{\mathbb{H}})\rightarrow (\widehat{\Omega},K_{\widehat{\Omega}})$.

Then by repeated use of the proof of Part (1), we deduce that $\Omega$ is not CAT(0).\\
(b). Assume that $\partial\Omega$ does not contain any affine disks at $\{\vec{0}\}$. Similarly we only need to check that $\mathbb{H}\times\Delta\times\{\vec{0}\}\subset\widehat{\Omega}$.

\bigskip
The proof of the theorem could be simplified if we use the following lemma.
\begin{lemma}[\cite{frankel1991applications}, Theorem 9.3]\label{frankel}\label{frankel}
Suppose $\Omega \subset \mathbb{C}^{d}$ is a $\mathbb{C}$-proper convex open set.
Suppose that $V \subset \mathbb{C}^{d}$ is a complex affine subspace intersecting $\Omega$ and $\{A_{n} \in Aff (V)\}$ is a sequence
of affine maps such that $A_{n}(\Omega \cap V)$ converges in the local Hausdorff topology to a
$\mathbb{C}$-proper convex open set $\widehat{\Omega}_{V} \subset V$. Then there exists affine maps $B_{n} \in$ Aff $\left(\mathbb{C}^{d}\right)$ such
that $B_{n} \Omega$ converges in the local Hausdorff topology to a $\mathbb{C}$-proper convex open set
$\widehat{\Omega}$ with $\widehat{\Omega} \cap V=\widehat{\Omega}_{V}$.
\end{lemma}

Now suppose $V=\mathbb{C}^2\times\{\vec{0}\}$ and $\Omega_{V}=\Omega\cap V$.
Let $G,\: W \subset \mathbb{R}$ and $U \subset \mathbb{C}$ be neighborhoods of 0 such that $f: G \times U \rightarrow W$ and
$$
\Omega_{V} \cap \mathcal{O}=\{(x+i y, z): x \in G, z \in U, y>f(x, z)\},
$$
where $\mathcal{O}=(G+i W) \times U$. By rescaling we may assume that $B_{1}(0) \subset U$.
We can find $a_{n} \rightarrow 0$ and $z_{n} \in B_{1}(0)$ such that $f\left(0, z_{n}\right)=a_{n}\left|z_{n}\right|^{n}$ and $f(0, w) \leq a_{n}|w|^{n}$ for all $w \in \mathbb{C}$
with $|w| \leq\left|z_{n}\right|$.

By the hypothesis that $\partial \Omega_{V}$ has no non-trivial complex affine disks, we obtain that $z_{n} \rightarrow 0$ and hence
$f\left(0, z_{n}\right) \rightarrow 0$. Passing to a subsequence we may assume that $\left|f\left(0, z_{n}\right)\right|<1 .$
Consider the sequence of linear transformations
$$
A_{n}=\left(\begin{array}{cc}{\frac{1}{f\left(0, z_{n}\right)}} & {0} \\ {0} & {z_{n}^{-1}}\end{array}\right), \quad n=1,2,\cdots,
$$
and $\Omega_{V_{n}}=A_{n}\Omega_{V}\rightarrow\widehat{\Omega}_{V}$. Therefore,
$$
\Omega_{V_{n}} \cap \mathcal{O}_{n}=\left\{(x+i y, z): x \in G_{n}, z \in U_{n}, y>f_{n}(x, z)\right\},
$$
where $G_{n}=f\left(0, z_{n} \right)^{-1} G, \: U_{n}=z_{n}^{-1} U$, and $\mathcal{O}_{n}=A_{n} \mathcal{O}$, and
$$
f_{n}(x, z)=\frac{1}{f\left(0, z_{n}\right)} f\left(f\left(0, z_{n}\right) x, z_{n} z\right).
$$
For $|w| < 1$, we then have
$$
f_{n}(0, w)=\frac{f\left(0, z_{n} w\right)}{f\left(0, z_{n}\right)} \leq \frac{a_{n}\left|z_{n}\right|^{n}|w|^{n}}{f\left(0, z_{n}\right)}=|w|^{n},
$$
which implies that
$$
\{0\}\times\Delta \subset \partial \widehat{\Omega}_{V}.
$$
By using
$$
\Omega_{V_{n}} \cap(\mathbb{C} \times\{0\})=\frac{1}{f\left(0, z_{n}\right)}(\Omega_{V} \cap(\mathbb{C} \times\{0\})),
$$
and $f\left(0, z_{n}\right)\rightarrow 0$,
we have $\mathbb{H}\times\{0\}\subset\widehat{\Omega}_{V}$.
Since $\widehat{\Omega}_{V}$ is convex, $\mathbb{H}\times\Delta\subset\widehat{\Omega}_{V}$ is valid.
It follows immediately from Lemma $\ref{frankel}$ that there exists $B_{n}\in Aff(\mathbb{C}^d)$ such that $B_{n}\Omega\rightarrow\widehat{\Omega}$ and $\mathbb{H}\times\Delta\times\{\vec{0}\}\subset\widehat{\Omega}$, which completes the proof.
\end{proof}

\bigskip
It's natural to ask whether the $m$-convexity is a sufficient condition for bounded convex domains being CAT(0). However, the following example shows that $m$-convexity does not imply CAT(0) in general.
\begin{exam}[\cite{Zimmer2019subellip}, \:Example 7.3]
Let $\Omega_{1}, \Omega_{2}$ be bounded strongly convex domains in $\mathbb{C}^2$ with $C^{\infty}$ boundary. Furthermore, we assume $0 \in \partial \Omega_{j}$, and the real hyperplane
$$
\left\{\left(z_{1}, z_{2}\right) \in \mathbb{C}^{2}: \operatorname{Re}\left(z_{j}\right)=0\right\}
$$
is tangent to $\Omega_{j}$ at $0$, and
$$
{\Omega_{j} \subset\left\{\left(z_{1}, z_{2}\right) \in \mathbb{C}^{2}: \operatorname{Re}\left(z_{j}\right)>0\right\}}.
$$
Define $\Omega=\Omega_{1}\cap\Omega_{2}$.
\end{exam}
Since each $\Omega_{j}$ has smooth boundary, we see that $(\epsilon, \epsilon) \in \Omega$ for $\epsilon>0$ sufficiently small. So $\Omega$ is non-empty. Furthermore, since each $\Omega_{j}$ is strongly convex, it follows that with a constant $C>0$
$$
\delta_{\Omega_{j}}(z ; v) \leq C \delta_{\Omega_{j}}(z)^{1 / 2}
$$
for all $1 \leq j \leq 2, \:z \in \Omega_{j},$ and non-zero $v \in \mathbb{C}^{2}$. Then, for $z \in \Omega$ and non-zero $v \in \mathbb{C}^{2}$, we have
$$
\delta_{\Omega}(z ; v)=\min _{1 \leq j \leq 2} \delta_{\Omega_{j}}(z ; v) \leq \min _{1 \leq j \leq 2} C \delta_{\Omega_{j}}(z)^{1 / 2}=C \delta_{\Omega}(z)^{1 / 2},
$$
from which we deduce that $\Omega$ is $2$-convex. However the set of domains $\{n\cdot\Omega\}$ converges in the local Hausdorff topology to the domain
$$
D=\left\{\left(z_{1}, z_{2}\right) \in \mathbb{C}^{2}: \operatorname{Re}\left(z_{1}\right)>0, \: \operatorname{Re}\left(z_{2}\right)>0\right\}.
$$
Thus $\partial D$ contains an affine disk.

Then, by repeated use of the proof of Theorem $\ref{main}$, we obtain that $(\Omega,K_{\Omega})$ is not CAT(0).

\bigskip
\bibliography{reference}
\bibliographystyle{plain}{}
\end{document}